%% file: irreducibility_of_the_equation_v1.tex
\documentclass[12pt]{article}
\usepackage{amsfonts}
 \usepackage{amsthm}
 \usepackage{amsmath}
 \usepackage{amscd}
 \usepackage{amssymb}
 \usepackage{latexsym}
 \usepackage{epsfig}
 
\pdfoutput=1
\usepackage[ascii]{inputenc}
\usepackage[T1]{fontenc}
\usepackage{amsmath,amssymb,amsfonts,textcomp,epsfig}
\usepackage{color}
\usepackage{calc}

\def\C{{\mathbb C}}

\def\Z{{\mathbb Z}}
\def\R{{\mathbb R}}

\newtheorem{theorem}{Theorem} 

\newtheorem{proposition}{Proposition}
\newtheorem{corollary}{Corollary}

\title{Irreducibility of the Picard-Fuchs equation related to the Lotka-Volterra polynomial $ x^2 y^2(1-x-y) $ }
 \author{Lubomir Gavrilov \\
 \normalsize \it Institut de Math\'{e}matiques de Toulouse, UMR 5219\\
 \normalsize \it Universit\'{e} de Toulouse, CNRS\\
   \normalsize \it UPS IMT, F-31062 Toulouse Cedex 9, \normalsize \it   France}
\begin{document}
\maketitle
\begin{abstract}
We prove that the Zarisky closure of the monodromy group of the polynomial $ x^2 y^2(1-x-y) $ is the symplectic group $Sp(4,\C)$. This shows that some previous results about this monodromy representation are wrong.
\end{abstract}
\tableofcontents

\section{Introduction.}

The study of a germ of a vector fields in $\R^4$ with two pairs of non-resonant imaginary eigenvalues reduces to the study of the special perturbations
\begin{equation}
\label{lvp}
xy(1-x-y) d \log H + \varepsilon_1 y dx +\varepsilon_2 yx^2 dx = 0 
\end{equation}
of the integrable quadratic  foliation
\begin{equation}
\label{lv}
xy(1-x-y) d \log H =  0
\end{equation}
with first integral $ H= x^py^q(1-x-y)$, e.g. \cite[chapter 1, section 4.6]{aais94}. Equivalently, we may consider 
 the following generalized Lotka-Volterra system associated to   (\ref{lv}) 
\begin{equation}
\label{glv}
\bigg \{
\begin{aligned}
x' =& x[q(1-x-y)-y]\\
y' = &y[x-p(1-x-y)]
\end{aligned}
\end{equation}
and the perturbed foliation (\ref{lvp}) is associated then to
\begin{equation}
\label{glvp}
\bigg \{
\begin{aligned}
x' =& x[q(1-x-y)-y]\\
y' = &y[x-p(1-x-y) + \varepsilon_1  +\varepsilon_2 x^2]  .
\end{aligned}
\end{equation}
The limit cycles of the perturbed system correspond  to the zeros of the displacement map 
$$
\Delta(h) = \frac1h \int_{H=h}x^{p-1}y^{q}( \varepsilon_1 +\varepsilon_2 x^2) dx + O(\varepsilon_1^2+\varepsilon_2^2) .
$$
where $\Gamma_h= \{H=h\}$ is a continuous family of ovals (closed orbits) of the non-perturbed system (\ref{glv}). The unicity of the limit cycle of (\ref{lvp}),(\ref{glvp}) was shown first  by \cite[Zoladek,1986]{zola87}, who proved the monotonicity of the function
$$
F(h) = \int_{H=h}  x^{p-1}y^q dx/  \int_{H=h} x^{p+1}  y^q  dx 
$$
on the maximal interval $(0,h_1)$ where the ovals of $\{ (x,y) : H(x,y)= h\}$ exist.
Note that the system (\ref{glvp}) has $\{x=0\}$ and $\{y=0\}$ as invariant lines. In this relation, recall that a
plane  quadratic vector field with an invariant line has a unique limit cycle (if any), as explained by
\cite[Coppel, 1989]{copp89}, but see also \cite[Zegeling and Kooij, 1994]{zeko94}. 

The case of a more general quadratic perturbations of (\ref{glv}) was studied also by \cite[Zoladek, 1994]{zola94} and revised recetly in \cite{zola15}.
The commont point of the above mentioned papers is, that they use \emph{ad hoc} methods based on apriori estimates. The essential reason why all these estimates hold remains hidden. 

In 1985  Van Gils and Horozov \cite{giho85} gave an overview of the varous methods, which have been used at this time to prove the uniqueness of the limit cycles for the perturbations of the generalized Lotka-Volterra system (\ref{glv}).  The central result of their paper is that in the particular case, in which $p=q$ is an integer, the functions
\begin{equation}
\label{abelian}
I(h)= \alpha \int_{H=h}  x^{p+1}y^p \, dx      + \beta   \int_{H=h}  x^{p-1}y^p \, dx, \; \alpha, \beta \in \R, h \in (0, h_1)
\end{equation}
satisfy a Picard-Fuchs equation of second order, whose coefficients are rational in $h^{1/p}$. The authors used then topological arguments as the Rolle's theorem, to bound the zeros of 
$I(h)$ in terms of the degrees of the coefficients of the Picard-Fuchs equation, from which the result of Zoladek \cite{zola87} follows.

More precisely, if $h_1>0$ and $h_2=0$ are the critical values of the Lotka-Volterra integral $ H= x^py^p(1-x-y)$, then
 the Abelian integral $I(h)$ (\ref{abelian}) allows an analytic continuation from $(0,  h_1)$ to a small neighbourhood of $h_1$. This follows from the Picard-Lefschetz formula and the fact, that the oval $\{ H= h \}$ represents a  cycle vanishing at $h_1$. At the other end of the interval, at $h=0$, the function is not analytic, but has a logarithmic type of singularity, as it follows from a generalised Picard-Lefschetz formula \cite{giho85}. Namely
$$
I(h)= J(z) \ln (z) + K(z), \; z= h^{1/p}
$$
for suitable functions $J(z), K(z)$, which are analytic in a neighbourhood of $z=0$. Note that $I(h) + 2\pi \sqrt{-1}  J(z)$ is an analytic continuation of $I(h)$ and hence it is an Abelian integral too. It has therefore a similar logarithmic type singularity at $h=h_1$.
Following  \cite{giho85}, denote by $W$ the Wronskian 
$$
W = 
\det \begin{pmatrix}
 \frac{d}{dz} I(z) &  \frac{d}{dz} J(z)\\
\frac{d^2}{dz^2} I(z) &  \frac{d^2}{dz^2} J(z)
\end{pmatrix}
$$
The Picard-Lefschetz formula then implies, that $W=W(h)$ as a function in $h$ allows an analytic continuation from the interval $(0,h_1)$ to a neighborhood of $ h_1$, and also that $W(h) $ allows an analytic continuation from the interval $(0,h_1)$ to a covering of a punctured neighbourhood of $h= 0$, in which  $W$ is analytic in $z=h^{1/p}$.
From this, the authors concluded that $W$ is in fact a rational function in $z=h^{1/p}$, and even computed it explicitly  \cite[Lemma 1]{giho85}. 

%

This conclusion, that $W$ is a rational function in $h^{1/p}$ is, however, wrong.
Indeed, there exist functions of moderate growth, analytic on the universal covering of $ \C~\setminus\{0, h_1\}$, with a branch on $(0,h_1)$ which are analytic at $h_1$, analytic in $h^{1/p}$ at $h=0$, but still not algebraic in $h$. To construct an example, consider a second order Fuchs  equation with singular points at $0, h_1,\infty$ and Riemann scheme
$$
\begin{pmatrix}
0& h_1& \infty\\
0& 0& \alpha \\
\frac1p& 0& \beta
\end{pmatrix}
$$
where $\alpha+\beta + 1/p= 1$. Let $I(h)$ be a (branch of a) solution on $(0,h_1)$, analyic in a small neighbourhood of $h=h_1$. For generic values of the  characteristic exponent $\alpha$ (or $\beta$) our equation  has no algebraic solutions. Therefore $I(h)$ can not be analytic in a neighbourhood of $h=0$ too. It is concluded that $I(h)$ is analytic in $h^{1/p}$ in a neighbourhood of $h=0$. Clearly, this argument points out a gap, but does not disproof the result of \cite{giho85}.

The purpose of the present note is to study 
 in more detail the monodromy representation of  the  Lotka-Volterra polynomial $H= x^py^p(1-x-y)$, in the case when  $p$ is an integer. The knowledge of this monodromy representation allows, according to 
 \cite{boga11},  to compute the minimal degree of the differential equation satisfied by $I(h)$.   
We shall show in this way, that in the first non trivial case $p=2$, the Abelian integral $I(h)$ satisfies a linear differential equation  of minimal degree four, even if the coefficients are supposed to be algebraic functions. 
\emph{Thus, the result of \cite[Lemma 1]{giho85} is definitely wrong.}  Note also, that the the computation of \cite[section 3.3]{boga11}, and in particular Corollary 4 there, are also wrong.

We prove in fact a more general result  about the attached Lie group $\mathbf G$, which is the Zarisky closure of the monodromy group of the polynomial  $H(x,y)$. Namely, we show that in the case $p=2$ the group $\mathbf G$ is isomorphic  to the symplectic group $Sp(4,\C)$, see Theorem \ref{case2}. As the standard representation of  $Sp(4,\C)$ is irreducible, then according to Corolary \ref{md} we obtain

\emph{The Abelian integral (\ref{abelian}) satisfies a Fuchs type equation of minimal degree four, even if its coefficients are supposed to be algebraic functions in $h$.
}

The paper is organized as follows. In the next short section we give some background, concerning the reduction of the degree of Picard-Fuchs operators. In section \ref{section3} we determine explicitely the monodromy operators, related to the two singular critical values of $H$. This (long) computation is contained in principal in \cite{boga11}, in the case of arbitrary integers $p,q$ and $H=x^py^q(1-x-y)$. In the particular case $p=q$, part of these computations simplify, and for this reason we give here an independent treatment. 

It is a straightforward observation, that the monodromy representation of $H=x^py^p(1-x-y)$ is reducible. We have in fact a two-dimensional plane $V_2$ of zero-cycles on which the monodromy acts as identity, as well a complementary $2p$-dimensional plane $V_{2p}$, invariant under the action of the monodromy group. The nature of this sub-representation $V_{2p}$ is studied in the last section \ref{section4} in the simplest non-trivial case $p=2$.  By taking the Zarisky closure of the monodromy group, we find that the sub-representation $V_4$ coincides with the standard representation of the symplectic group $Sp(4,\C)$.  This, combined with section \ref{section3} implies the claims about the degree of the Picard-Fuchs equation.

\section{Reduction of the degree of Picard-Fuchs equations.} 
\label{section3}
In this section we summarize, following \cite{boga11}, the necessary facts about the reduction of the degree of Picard-Fuchs differential operators.

To a non-constant polynomial $f\in \mathbb C[x,y]$ we associate its monodromy representation
$$
\pi_1(\C \setminus S, b) \to Aut(H_{1}(f_b, \mathbb Z))
$$
where $f_t= f^{-1}(t) \subset \C^2$
are the fibers of the fibration $f: \C^2 \to \C$.
The image of the fundamental group $\pi_1(\C \setminus S, b)$ is the monodromy group $\mathcal M$. The Zarisky closure $\mathbf G = \overline{\mathcal M}$ of $\mathcal M$ is a linear algebraic group, embedded in $GL_d(\mathbb C)$, $d= \dim H_{1}(f_b, \mathbb Z) $. It is nothing but the differential Galois group of a generic Picard-Fuchs system, related to the fibration defined by $f$. In the sequel, an important role is played by the connected component $\mathbf G^0$ of $\mathbf G$, containing the identity transformation.

Let $\delta(t)\subset H_1(f_t, \Z)$ be a continuous family of cycles, and $\omega$ a polynomial one-form. The Abelian integral 
$$I(t)= \int_{\delta(t)} \omega$$
satisfies a Picard-Fuchs equation of minimal degree $d$
$$
I^{(d)} + a_1 I^{(d-1)} + \dots a_d I = 0
$$
whose coefficients are rational functions in $t$, $a_i\in \C(t)$. Let $V\subset H_{1}(f_b, \C)$ be a vector plane, invariant under $\mathbf G$. Then obviously $d\leq \dim V$. Let $\gamma \in V$ and $\gamma(t)$ the corresponding continuous family of cycles (a locally constant section of the homology bundle). The set of such $\gamma$ with the property $\int_{\gamma(t)} \omega \equiv 0$ is an invariant sub-plane of $V$ which we denote by $V_1$. We conclude that $d= \dim V - \dim V_1$.

Consider now the following reduction problem.:

\emph{Find a differential equation
\begin{equation}
\label{md}
I^{(d^0)} + b_1 I^{(d^0-1)} + \dots b_{d^0} I = 0
\end{equation}
of minimal degree $d^0 \leq d$, such that $b_i=b_i(t)$ are algebraic functions in $t$. }

The computation of the degree $d_0$ goes along the same lines as above, except that the Lie group $\mathbf G$ is replaced by $\mathbf G^0$. Namely, let $V^0\subset H_{1}(f_b, \C) $ be a  plane, invariant under $\mathbf G^0$, that is to say a sub-representation of $\mathbf G^0$. Such a plane was called virtually invariant in \cite{boga11}.  It follows that $d_0 \leq \dim V^0$. 
 Let $\gamma \in V^0$ and $\gamma(t)$ the corresponding continuous family of cycles as above. The set of such $\gamma$ with the property $\int_{\gamma(t)} \omega \equiv 0$ is a virtually invariant sub-plane of $V^0$ which we denote by $V_1^0$. We conclude that 
 \begin{theorem}
  $d^0= \dim V^0 - \dim V_1^0$.
 \end{theorem}
 \begin{corollary}
 If the representation of $\mathbf G^0$ on $V^0$ is irreducible, then the minimal degree of the differential operator (\ref{md}) equals  $\dim V^0$.
 \end{corollary}
 See \cite{boga11} for proofs.

\section{Topology of the fibration, defined by the polynomial  $x^p y^p(1-x-y)$ .}
\label{section4}
Denote
 $$f= x^py^p(1-x-y)$$ for some $p\in \mathbb{N}^*$. The polynomial $f$ has two critical values $t_1= \frac{p^{2p}}{(2p+1)^{2p+1} }$ and $t_2=0$ 
 and non-isolated critical points along the lines $x=0$ and $y=0$.
For $t\neq t_{1,2}$ the algebraic curve $\Gamma_t = \{(x,y)\in \C^2: f(x,y)= t\}$ is a smooth Riemann surface and its fundamental group has $2p+2$ generators, see Fig.\ref{fig1}. As $\Gamma_t$
has three punctures (at infinity), then it is a genus $p$ algebraic curve. 
 We wish to describe the "continuous variation" of $\Gamma_t$ when $t$ varies along closed circuits in  $\C\setminus \{t_1,t_2\}$, or equivalently,the topology of the fibration
 \begin{align*}
f:\C^2 & \to \C\setminus \{t_1,t_2\} \\
(x,y) &\mapsto  f(x,y)=x^py^p(1-x-y)
\end{align*}
The calculation of this  geometric monodromy is in general a difficult task, see the survey of Siersma \cite{sier00}.

To begin with, we first localize $f$ at the singular points $(0,1)$ and $(1,0)$ and obtain a germ of analytic function with non-isolated critical points.
We study first their geometric monodromy, which is straightforward.
\subsection{The germ   $(x+\dots)^p(y+\dots)$. }
\label{toy}
Let $f: \C^2,0 \to \C,0$ be a germ of analytic function, such that
$$
f_0= \{ (x,y)\in \C^2 : f(x,y)=0 \}
$$
defines a germ of  a divisor with simple normal crossing and multiplicities $1$ and $p$. In appropriate coordinates in a suitable neighborhood  of the origin we have
$f(x,y)=x^py$, which will be assumed until the end of this section.

The  \emph{marked  fibration } associated to $f$ is by definition the usual fibration
\begin{align*}
f:\C^2 & \to \C\setminus 0 \\
(x,y) &\mapsto  f(x,y)=x^py
\end{align*}
whose fibers  are the Riemann surfaces  (topological cylinders)
$$f_t=\{(x,y): x^py=t\}  
$$ 
with $p+1$ marked points corresponding to the intersection of $f_t$ with the two fixed lines $\{x=1\}$ and $\{y=1\}$
$$S_t= \{ (1,t), (t^{1/p}e^{2k\pi i/p},1), k=0,1\dots,p-1 \} .$$
Continuous deformation of $t$ induces an isotopy of the marked fibers $f_t$. A continous variation of $t$ along a closed circuit about the origin induces therefore a diffeomorphism (geometric monodromy)
$$
M: f_t \to f_t, t\neq  0
$$
defined up to an isotopy, which permutes the marked points. 
We wish to describe  $M$
along the same lines, as in the case of an isolated   singularity of Morse type, $p=1$, e.g. \cite{agv88v2}. 

It is convenient to consider the 
 fundamental groupoid  $\pi_1(f_t,S_t)$ of the pair $(f_t,S_t)$ \cite{brow06}, which replaces the common fundamental group $\pi_1(f_t,*)$ in the case $p=1$. Recall that this groupoid  is the set of homotopy classes of loops $\gamma: [0,1] \to f_t$ such that $\gamma(0), \gamma(1) \in S_t$. Two loops $\gamma_1, \gamma_2$ are composable, if $\gamma_2(1)=\gamma_1(0)$ and in this case the homotopy class of $\gamma_1 \circ \gamma_2 \in \pi_1(f_t,S)$ is well defined. In the case when $S_t=*$ is a single point $\pi_1(f_t,S_t)= \pi_1(f_t,*) $ is the usual fundamental group of the fiber $f_t$.
 The fundamental groupoid $\pi_1(f_t,S)$ is freely generated by  $p+1$ loops, as shown on fig.\ref{fig2} in the particular case $p=3$.

 The geometric monodromy $M$ defines a  homomorphisms 
$$m_*: \pi_1(f_t,S_t) \to \pi_1(f_t,S_t) . 
$$
 Clearly, $M$ permutes cyclically the points
$ (t^{1/p},1)$ and fixes $(t,1)$. Define a loop $\gamma_0$ connecting $(t,1)$ to a point in the set $ \{(t^{1/p},1)\}$ and then inductively
\begin{equation}
\label{gk}
\gamma_k= m^k _* \gamma_0, k=0,1, ..., p .
\end{equation}
\begin{proposition}
\label{paths}
The $p+1$ loops $\gamma_0, \gamma_1, \dots, \gamma_p$ generate the groupoid $\pi_1(f_t,S_t)$, the closed loop $\alpha =\gamma_0^{-1} \circ \gamma_p$ generates the fundamental group $\pi_1(f_t,*)=\Z$.
 (see fig. \ref{fig2}). 
\end{proposition}
\begin{proof}
The projection $(x,y) \to x$ maps $f_t$ isomorphically to $\C^*$, and the images of the marked points are
$$
1, t^{1/p}e^{2k\pi i/p}, k=0,1\dots,p-1 .
$$
Alternatively, in the ball $B_R= \{(x,y)\in \C^2: |x|^2+|y|^2 \leq R^2 \}$, for sufficiently small $|t|$ the surface
$f_t\cap B_R$
is projected under $(x,y)\to x$ to the annulus shown on  fig.\ref{fig2}.
When $t$ makes one turn around the origin, in a clockwise direction,  the marked points $t^{1/p}e^{2k\pi i/p}$ permute cyclically, and the corresponding paths $\gamma_k$ connecting $1$ to $t^{1/p}e^{2k\pi i/p}$ are as on fig. \ref{fig2}. 
Note that in the case $p=1$ we get the usual Picard-Lefschetz formula.
\end{proof}
For a further use, note the following
\begin{corollary}
\label{cor1}
It is always possible to choose the initial loop $\gamma_0$ in such a way, that $\gamma_i$, $0\leq i \leq p$, are non-intersecting, and non self-intersecting loops. In this case the union $\cup_{i=0}^p \gamma_i\subset f_t$ is an embedded planar graph, which is a deformation retract of the marked cylinder $f_t$.
\end{corollary}

Consider finally the  relative homology group  $H_1(f_t,S_t)= H_1(f_t,S_t,\Z)$ which is isomorphic as a $\Z$-module to $ \Z^{p+1}$. The generators of 
$H_1(f_t,S_t)$ are represented by the paths $\gamma_0, \gamma_1, \dots, \gamma_p$ defined in Proposition \ref{paths}. We conclude that the corresponding equivalence classes of paths 
 $$
[\gamma_0], [\gamma_1], \dots, [\gamma_{p-1}],  [\gamma_p] = [\alpha] +  [\gamma_0] \in H_1(f_t,S_t) .
$$
define a basis   of  the relative homology group.

\begin{figure}
\begin{center}
\def\svgwidth{0,5\columnwidth}
 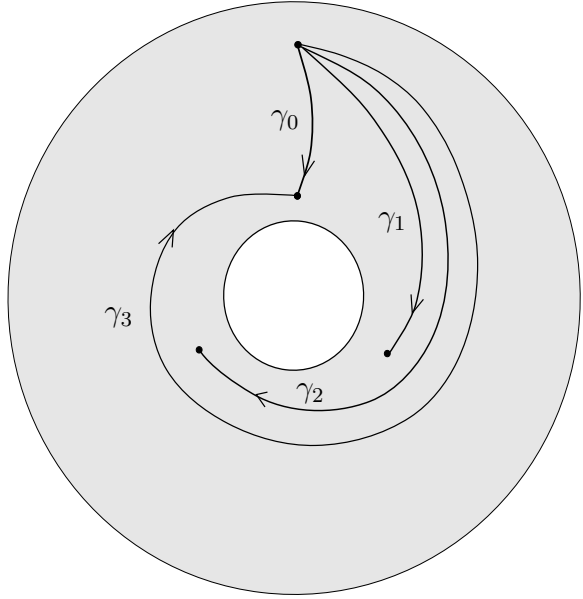
 \end{center}
 \caption{The fundamental groupoid $\pi_1(f_t,S)$ of the cylinder $f_t = \{x^3 y = t\}$ with four marked points.}
  \label{fig2}
 \end{figure}
The diffeomorphism $M=f_t \to f_t$ induces a homomorphism (Picard-Lefschetz monodromy operator)
$$M_* :H_1(f_t,S_t) \to H_1(f_t,S_t). $$
 Note that $M_*[\alpha]=[\alpha]$ and in the basis
$$
[\gamma_0], [\gamma_1], \dots, [\gamma_{p-1}], [\alpha] 
$$
$M_*$ is represented  by the matrix
$$
\left(\begin{array}{cccc|c}
0 & \dots & 0 & 1 & 0 \\
\hline
1 & \dots & 0 & 0 & 0 \\
\vdots & \ddots & \vdots & \vdots & \vdots \\
0 & \hdots & 1 & 0 & 0
\\0 & \hdots & 0 & 1 & 1\end{array}\right)
$$
with characteristic polynomial 
$$
(-1)^{p+1}(\lambda-1) (\lambda^p - 1) .
$$

 \subsection{The topology of the fiber  $f_t= \{x^py^p(1-x-y)=t\}$ when $t$ is close to $t_2=0$.}
 Our aim in this section is to construct  explicit generators for the fundamental group of $f_t$ which allows a simple description of the action of the monodromy transformations of the fibration, defined by $f$.
 The generators will be presented in the form of an embedded graph, which is a deformation retract of $f_t$ (see Corollary \ref{cor1}).
 \begin{figure}
\begin{center}
\def\svgwidth{0,4\columnwidth}
 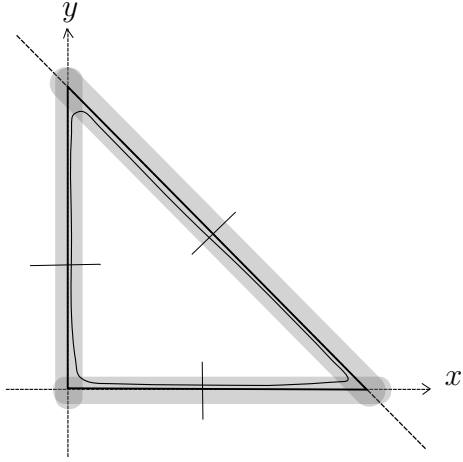
 \end{center}
 \caption{Tubular neighborhood of the triangle, with vertices at the singular points $(0,0), (1,0), (0,1)$ in $\R^2$.}
 \label{fig4}
 \end{figure}
 
 Namely, consider the real triangle, with vertices  the singular points $(0,0), (1,0), (0,1)$ in $\R^2$. Let $U\subset \C^2$ be a suitable tubular neighborhood of this triangle. We suppose that $\partial U$ is transversal to the complex lines
 $$
 \{ x=0, y=0, x+y+1=0 \} \subset \C^2 
 $$
and moreover $f_0\cap U$ is a deformation retract of $f_0$. 
It follows that for  sufficiently small $|t|$ the border  $\partial U$ is transversal also to $f_t$, and that $f_t\cap U$ is a deformation retract of $f_t$. This allows to localize our description of the fiber $f_t$ near the triangle with vertices $(0,0), (1,0), (0,1)$, see fig. \ref{fig4}.

Consider three cross-sections (complex discs) transversal to the sides, dividing the triangle into three pieces, with corresponding tubular neighborhoods $U_{13}, U_{23}, U_{12}$, where  $U= U_{13}\cup U_{23} \cup U_{12}$. The cross sections intersect the fiber $f_t\subset \C^2$  for sufficiently small
non-zero  $|t|$ in  exactly $p$, $p$ and $1$ points respectively, which will be the marked points from the  section \ref{toy}. Let $U_{12}$ be the tubular neighborhood, containing $(0,0)$. Then $f_t\cap U_{12}$ has $p$ connected components (topological cylinders) which coincide with the fibers of a Morse polynomial, and it is retracted to $p$ disjoint segments.
The  fibers $f_t\cap  U_{13}$ and $f_t\cap  U_{23}$ are described as in the section \ref{toy}, Corollary \ref{cor1}.

We construct now a graph, embedded in $f_t\cap U$ which is a deformation retract of $f_t$. For this purpose we take together the corresponding graphs near the singular point $(0,0), (1,0), (0,1)$, constructed in section \ref{toy}. The assembling of these graphs is shown on fig. 3 in the case $p=3$, the general case of arbitrary $p$ being analogous. It is easy to check (by making use of a partion of the unity) that the resulting graph is embedded in $f_t$ and is a deformation retract of $f_t$. 
\begin{figure}
\begin{center}
 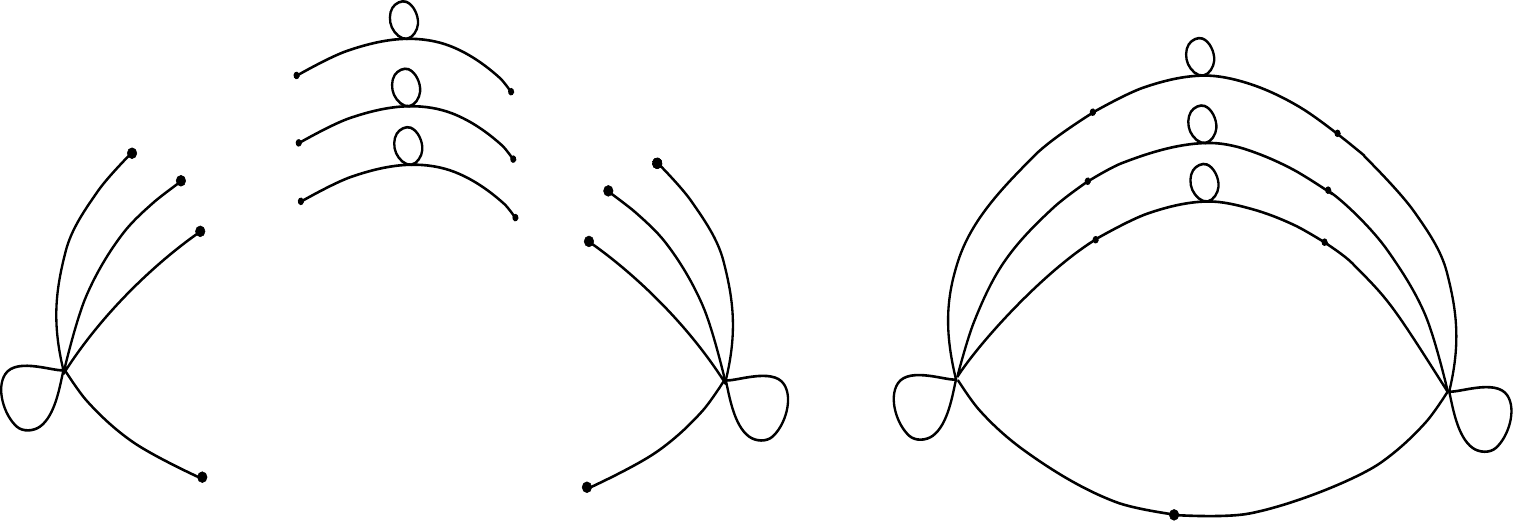
 \end{center}
 \caption{Assembling of a deformation retract of the fiber $f_t=\{x^3y^3(1-x-y)=t\}$.}
 \label{fig3}
 \end{figure}

\subsection{ The monodromy of the fibers $f_t$ along closed paths near the singular value $t=0$.}
The next step is to describe the action of the monodromy  transformation $M : f_t\to f_t$, corresponding to the singular value $t=0$,  which once again easily follows from the preceding constructions. More specifically, we shall describe the linear operator $M_*\in Aut(H_1(f_t,\Z))$. We begin by choosing a suitable basis of $H_1(f_t,\Z)$. In a suitable neighborhood of $(x,y)=(0,0)$ the fiber $f_t= \{x^py^p(1-x-y)=t \}$ has $p$ connected components homeomorphic to cylinders, and each component carries a vanishing cycle denoted 
$$\delta_{12}^k, k=0, 1,\dots,p-1 . $$
At $(0,1)$ and $(1,0)$ we have vanishing cycles denoted respectively $\delta_{13}$ and $\delta_{23}$. We have therefore $p+2$ vanishing cycles, additional $p$ cycles are constructed as follows
$$
\delta_k= M^k \delta_0, k=0,1,\dots,p-1
$$
where $\delta_0=\delta(t)$ is represented by the oval of $f_t$ for $t_1<t<t_2$. The $2p+2$ cycles which we described are independent in $H_1(f_t,\Z)$. 
\begin{proposition}
The monodromy operator $M_*$ acts as follows
$$
\boxed{\delta_{12}^0 \to \delta_{12}^{1} \to \dots \to  \delta_{12}^{p-1} \to \delta_{12}^0 , \;\; \delta_{13} \to \delta_{13},  \delta_{23} \to \delta_{23}}
$$
and
$$
\boxed{
\delta_0  \to \delta_{1}\to \dots \to \delta_{p-1},\;\; \delta_{p-1} \to \delta_0 - \delta_{12}^0 - \delta_{13}- \delta_{23} .}
$$
\end{proposition}
\begin{proof}
Having described the monodromy of the fiber $f_t$ localized around the singular points of $f$, everything is obvious, except $M_*\delta_{p-1}$.
To  compute $M_*\delta_{p-1}= M_*^p \delta_0$ we use the action of $M^p$ described in the relative homology of the local fibers. As noted in \cite{giho85} (or see fig. \ref{fig2}) $M^p$ is a Picard-Lefschetz operator, and
$$
 M^p_*\delta_0= \delta_0 - \delta_{12}^0 - \delta_{13}- \delta_{23} .
$$
\end{proof}
 It is straigthtforward to check that the characteristic and the minimal polynomials of $M_*$  are equal respectively to
 $$
 (\lambda-1)^2 (\lambda^p-1)^2 , (\lambda-1) (\lambda^p-1)^2
 $$
\begin{corollary}
The orbit of the cycle $\delta_0$ under the action of $M_*$ spans the $2p$-dimensional  vector space generated by
$$
\delta_i, \delta_{12}^i+\delta_{13}+\delta_{23}, \; i=0,1,\dots, p-1 .
$$
\end{corollary}

\subsection{The monodromy of the fibration  $f: \C^2 \to \C\setminus \{t_1,t_2\} $ .}

The monodromy group is generated by two operators, $M_1$ and $M_2=M$ related to the singular values $t_1, t_2=0$.  The value $t_1$ corresponds to a Morse critical point and the operator $M_1$ is given by the Picard-Lefschetz formula. Its description amounts to compute the intersection form on $H_1(f_t,\Z)$.
\begin{proposition}
\label{intersections}
The intersection numbers of the  $2p+2$ cycles
$$
 \delta_{13},  \delta_{23}, \delta_i,  \delta_{12}^i, 0\leq i,j\leq p-1
$$
 generating $H_1(f_t,\Z)$ are as follows 
$$
(\delta_{12}^i \cdot \delta_{12}^j)=0, (\delta_{12}^i \cdot \delta_{13}) = 0, (\delta_{12}^i \cdot \delta_{23}) = 0, (\delta_{13} \cdot \delta_{23})=0
$$
$$
(\delta_i \cdot  \delta_{13})= (\delta_i \cdot  \delta_{23})=1
$$
$$
(\delta_i\cdot \delta_{12}^i) = 1; (\delta_i \cdot \delta_{12}^j) = 0, i\neq j
$$
$$
(\delta_i \cdot \delta_j)= -1, i<j .
$$
\end{proposition}
\begin{proof}
\begin{figure}
\begin{center}
\def\svgwidth{0,7\columnwidth}
 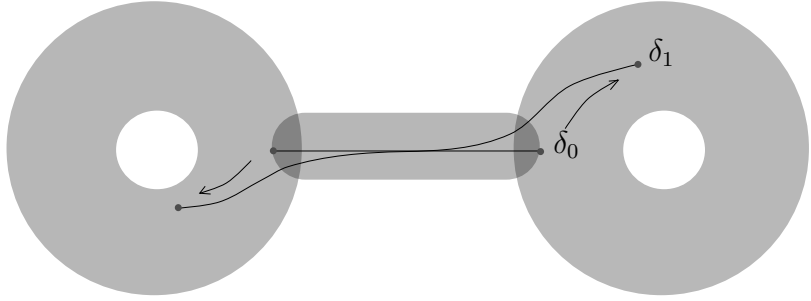
 \end{center}
 \caption{The biholomorphic image of $f_t$ on the line $1-x-y=0$ in $U_{12}$}
 \label{fig5}
 \end{figure}
The first and the third lines are obvious.  For the second, we choose an orientation on $\delta_{13}, \delta_{23}$ in such a way, that
$$
(\delta_0 \cdot  \delta_{13})= (\delta_0 \cdot  \delta_{23})=1
$$
and then use the invariance of the intersection number under the action of the monodromy. The only non-trivial fact is the fourth line. We note that according to fig. \ref{fig3}, the intersection number of the relative cycles $\delta_i\cap U_{12}$, $\delta_j\cap U_{12}$ are well defined and equal $0$.
 The fibration defined by $f$ in $U_{23 }\cup U_{13}$ can be further continuously deformed in a way, which does not change the topology of the fibers and their monodromy. Namely, by such a deformation we may replace $f$ by 
$$
\tilde{f} = (1-x-y)x^p(1-x))
$$
and consider the linear projection $$\pi: U_{23}\cup U_{13} \to U_{23}\cup U_{13} \cap \{1-x-y=0\}$$
parallel to the lines $x=const.$ As in the section \ref{toy}, the projection 
$$
\pi : U_{23}\cup U_{13}\cap f_t \to U_{23}\cup U_{13} \cap \{1-x-y=0\}
$$
is an injectif local biholomorphism, and its image is shown on fig. \ref{fig5}. When $t$ makes one turn around the origin, the two marked points corresponding to the ends of the relative cycle $\delta_0 \cap U_{23}\cup U_{13}$ turn in the same direction. The result is the relative cycle  $\delta_1 \cap U_{23}\cup U_{13}$ shown on fig. \ref{fig5}. This already proves that 
$$(\delta_0 \cdot \delta_i)= (\delta_0 \cdot \delta_j) = \pm 1, 1\leq i,j \leq p-1$$
and by invariance of the intersection form the numbers 
$$(\delta_i \cdot \delta_j), 1\leq i< j \leq p-1
$$
are all equal to the either $+1$ or to $-1$.
We have finally
\begin{align*}
(\delta_0 \cdot \delta_1)= & (M^{p-1}_*\delta_0 \cdot M^{p-1}_* \delta_1) =  (\delta_{p-1} \cdot M^{p}_* \delta_0) \\
= &(\delta_{p-1} \cdot   (\delta_0 - \delta_{12}^0 - \delta_{13}- \delta_{23}) )= - (\delta_{0} \cdot   \delta_{p-1}) - 2\\
=& -(\delta_0 \cdot \delta_1) -2
\end{align*}
and hence $(\delta_0 \cdot \delta_1)=-1$.
\end{proof}

\section{A case study : the polynomial  $x^2y^2(1-x-y)$.}
In this section we consider in detail the first non-trivial case $p=2$, in which the fibers $f_t$ are genus two Riemann surfaces with three punctures. 
For definiteness, denote $M_1, M_2=M \in Aut(H_1(f_t,\Z))$ the monodromy operators associated to simple closed loops around $t_1$ or respectively $t_2=0$. 
The monodromy group $\mathbf M$ of the polynomial $x^2y^2(1-x-y)$ is then the subgroup of $Aut(H_1(f_t,\Z))$ generated by $M_1, M_2$. The smallest algebraic variety containing $\mathbf M$ is an algebraic group, denoted $\mathbf G$. It is the Zarisky closure of $\mathbf M$. 

We note that $(H_1(f_t,\Z))$ carries a (degenerate) intersection form $\omega$ of rank $p=2$ invariant under the action of $M_1, M_2$. It is easily 
verified, that $\mathbf M$ and hence $\mathbf G$  is isomorphic to a subgroup of the symplectic group $Sp(4,\C)$.

We shall prove the following
\begin{theorem}
\label{case2}
The Zarisky closure of the monodromy group of the polynomial $x^2y^2(1-x-y)$ is isomorphic to the symplectic group $Sp(4,\C)$.
\end{theorem}
To the end of the section we give the proof of this remarkable fact. 
A basis of the first homology group $H_1(f_t,\Z)$ will be chosen as in the preceding section
\begin{equation}
\delta_0, \delta_1, \delta_{12}^0, \delta_{12}^0, \delta_{13}, \delta_{23}
\label{basis}
\end{equation}
where $\delta_0=\delta_0(t)$ is a cycle vanishing at the unique Morse critical point when $t$ tends to $t_1$, $\delta_{13}, \delta_{23}$ are vanishing cycles at the singular points $(1,0), (0,1)$, and $\delta_{12}^0, \delta_{12}^0$ are cycles vanishing at $(0,0)$.
The cycle $\delta_1$ is the image of $\delta_0$ under the action of the monodromy operator about the singular value $t=0$, $\delta_1=M_*\delta_0$. The cycles (\ref{basis}) are represented by closed loops on the Riemann surface $f_t$, and
by abuse of notation we denote these loops by the same letter.  The closed loops can be chosen in a way that their union is a deformation retract of $f_t$,  see fig.\ref{fig1}.

\begin{figure}
\begin{center}
 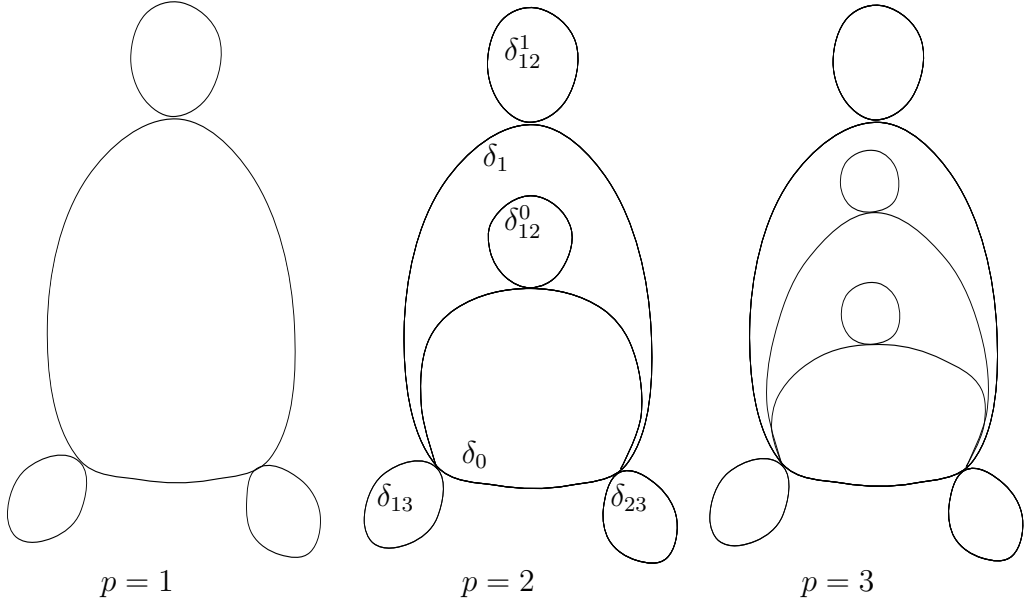
 \end{center}
 \caption{The canonical basis of generators of the fundamental group of the fiber $f_t$.}
 \label{fig1}
 \end{figure}
 To compute $M_1$ we note that according to Proposition \ref{intersections}
the sign of intersection indexes of the cycles of $H_1(f_t,\Z)$ can be chosen
 as follows
\begin{center}
\begin{tabular}{|c|c|c|c|c|c|c|}\hline 
& $\delta_0$ &$ \delta_1 $& $\delta_{12}^0 $& $\delta_{12}^1 $& $\delta_{13} $& $\delta_{23}$ \\\hline $\delta_0$ & 0 & -1 & 1 & 0 & 1 & 1 
\\\hline $\delta_1$ & 1 & 0 & 0 & 1 & 1 & 1 \\\hline $\delta_{12}^0 $ & -1 & 0 & 0 & 0 & 0 & 0 \\\hline $\delta_{12}^1 $ & 0 & -1 & 0 & 0 & 0 & 0 \\\hline $\delta_{13} $& -1 & -1 & 0 & 0 & 0 & 0 \\\hline $\delta_{23} $ & -1 & -1 & 0 & 0 & 0 & 0 \\\hline \end{tabular}
\end{center}
The monodromy operators $M_1$ and $M_2^2$  in this  basis (\ref{basis})  are represented by the following matrices (denoted by the same letter).
$$
M_1=\left(\begin{array}{c|ccccc} 1& -1 & 1 & 0 & 1 & 1 \\\hline0 & 1 & 0 & 0 & 0 & 0 \\0 & 0 & 1 & 0 & 0 & 0 \\0 & 0 & 0 & 1 & 0 & 0 \\0 & 0 & 0 & 0 & 1 & 0 \\0 & 0 & 0 & 0 & 0 & 1\end{array}\right), \;\;
M_2=\left(\begin{array}{cc|cc|cc}0 & 1 & 0 & 0 & 0 & 0 \\1 & 0 & 0 & 0 & 0 & 0 \\\hline0 & -1 & 0 & 1 & 0 & 0 \\0 & 0 & 1 & 0 & 0 & 0 \\\hline0 & -1& 0 & 0 & 1 & 0 \\0 & -1 & 0 & 0 & 0 & 1\end{array}\right)
$$
The homology group $H_1(f_t,\Z)$ splits into two invariant subspaces under $\mathbf{M}$ 
$$
H_1(\Gamma_h,\C)= V_1\oplus V_2$$
where
$$
V_1=Span \{\delta_0, \delta_1, \delta_{12}^0+\delta_{13}+ \delta_{23},  \delta_{12}^1 + \delta_{13}+ \delta_{23} \}, V_2=Span \{\delta_{13}- \delta_{23} , \delta_{13}+ \delta_{23}- 2 \delta_{12}^0 - 2 \delta_{12}^1\} .
$$
The monodromy group $\mathbf{M}$ acts on $V_2$ as the identity transformation, and on $V_1$ in the basis
$
\delta_0, \delta_1, \delta_{12}^0+\delta_{13}+ \delta_{23},  \delta_{12}^1 + \delta_{13}+ \delta_{23}$
 the monodromy operators are represented by the following matrices (which we denote by the same letters)
$$
M_1=\left(\begin{array}{cccc}1 & -1 & 3 & 2 \\0 & 1 & 0 & 0 \\0 & 0 & 1 & 0 \\0 & 0 & 0 & 1\end{array}\right),
M_2=\left(\begin{array}{cccc}0 & 1 & 0 & 0 \\1 & 0 & 0 & 0 \\0 & -1 & 0 & 1 \\0 & 0 & 1 & 0\end{array}\right).
$$
Let $\mathfrak g $ be the Lie algebra of $\mathbf G$, that is to say the tangent space $T_{I} \mathbf G$ of the variety $\mathbf G$ at the identity matrix $I$.
Clearly $\frak g $  is isomorphic to a sub-algebra of $sp(4,\C)$ and to prove Theorem \ref{case2}
 it will be enough to check that $\frak g $ is isomorphic to $sp(4,\C)$. For this let us note first that
 $$
 M_1^k = \left(\begin{array}{cccc}1 & -k & 3k & 2k \\0 & 1 & 0 & 0 \\0 & 0 & 1 & 0 \\0 & 0 & 0 & 1\end{array}\right) \in \mathbf M, \forall k\in \mathbb Z $$
 which implies 
 $$
 M_1^z = \left(\begin{array}{cccc}1 & -z & 3z & 2z \\0 & 1 & 0 & 0 \\0 & 0 & 1 & 0 \\0 & 0 & 0 & 1\end{array}\right)  \in \mathbf G, \forall z\in \mathbb C 
 $$
 and hence
 $$
a= \left(\begin{array}{cccc}0 & -1 & 3 & 2 \\0 & 0 & 0 & 0 \\0 & 0 & 0 & 0 \\0 & 0 & 0 &0 \end{array}\right) \in \mathfrak g
$$
Similarly
$$
M_2^2=\left(\begin{array}{cccc}1 & 0 & 0 & 0 \\0 & 1 & 0 & 0 \\-1 & 0 & 1 & 0 \\0 & -1 & 0 & 1\end{array}\right), M_2M_1M_2^{-1}=
\left(\begin{array}{cccc}1 & 0 & 0 & 0 \\1 & 1 & 2 & 3 \\0 & 0 & 1 & 0 \\0 & 0 & 0 & 1\end{array}\right)
$$
and taking powers of these matrices we conclude that
$$
c=\left(\begin{array}{cccc}0 & 0 & 0 & 0 \\0 & 0 & 0 & 0 \\-1 & 0 & 0 & 0 \\0 & -1 & 0 & 0\end{array}\right), 
b=\left(\begin{array}{cccc}0 & 0 & 0 & 0 \\1 & 0 & 2 & 3 \\0 & 0 & 0 & 0 \\0 & 0 & 0 & 0\end{array}\right)
$$
belong to $\mathfrak g$. We shall check that in fact $\mathfrak g$ is generated as a Lie algebra by $a,b,c$. 

\begin{proposition}
The Lie algebra $\mathfrak g$ generated by the matrices $a,b,c$ is isomorphic to the symplectic algebra $\mathfrak sp (4,\C)$.
\end{proposition}
\begin{proof}
As  $\mathfrak g \subset \mathfrak sp (4,\C)$ it is enough to compute the Cartan decomposition of $ \mathfrak sp (4,\C)$ with respect to the intersection form  on $V_1\subset H_1(f_t,\Z)$, 
and verify that the basis of the decomposition belongs to $\mathfrak g$.
Note first that 
$$
[[a,b],a]= - 2 a, [[a,b],b] = 2 b
$$
and hence the matrices  $a,b,[a,b]$ 
generate   $\mathfrak sl_2(\mathbb C)$. This suggests that the matrix $H_1=[a,b]$ belongs to the Cartan subalgebra $\mathfrak h \subset \mathfrak g$. To find a second element $H_2$ of $\mathfrak h$ we compute  (some)  eigenvectors of $\text{ad}_{H_1} : \mathfrak g \to \mathfrak g$ , where  $\text {ad}_{H_1}(X)= H_1 X- X H_1$, until finding an appropriate candidate for $H_2$, after what we display the various root spaces. Namely, let
$$
X_{21}= -3 [a,b] + [a,c] - 4 a, X_{12}=  3 [a,b] + [b,c]+4b , H_2= [ X_{21}, X_{12}] .$$
Then $\mathfrak h = <H_1,H_2>$ is the Cartan subalgebra and let
$$
\lambda_1, \lambda_2 \in \mathfrak h^*, \lambda_1(H_1)=1,  \lambda_1(H_2)=-5,  \lambda_2(H_1)=0,  \lambda(H_2)=5 .
$$
be a basis of the dual space $\mathfrak h^*$. 
Then it is straightforward to check that
$\pm\lambda_1\pm\lambda_2$ are roots with corresponding one-dimensional roots spaces $\mathfrak g_{\pm\lambda_1\pm\lambda_2}$ 
spanned  by the vector  on the second line of Table \ref{table},
where
 $$
 Y_{12}= [X_{21},b], Z_{12}= [X_{12},a], U_1= b, V_1= a, U_2=  [[X_{21},b],X_{21}], V_2=  [ X_{12},[ X_{12},a]].
 $$
\begin{table}[h]
\begin{center}
\begin{tabular}{|c|c|c|c|c|c|c|c|}
\hline
$\mathfrak g_{\lambda_1-\lambda_2}$&$\mathfrak g_{-\lambda_1+\lambda_2}$&$\mathfrak g_{ \lambda_1+\lambda_2}$ &$\mathfrak g_{-\lambda_1-\lambda_2}$&
$\mathfrak g_{2\lambda_1}$&$\mathfrak g_{-2\lambda_1}$&$\mathfrak g_{2\lambda_2}$ &$\mathfrak g_{-2\lambda_2}$\\
\hline
$X_{12}$&$X_{21}$&$Y_{12}$ &$Z_{12}$&
$U_1$&$V_1$&$U_2$ &$V_2$\\
\hline
\end{tabular}
\end{center}
\caption{Root spaces of $\mathfrak g$}
\label{table}
\end{table}%

\end{proof}

\section{Concluding remarks}
Let $f=f(x,y)$ be an arbitrary non-constant polynomial. The set ${\cal A} $ of its  non-regular values is finite and therefore we can consider the monodromy representation of the fundamental group $\pi_1(\C \setminus  \mathcal A ,*)$ on $H_1(f^{-1}(t), \Z)$. 
  Cearly the representation  preserves the intersection form of the first homology group $H_1(f^{-1}(t), \Z)$. 
 
 The subplane $V_0 \subset H_1(f^{-1}(t), \Z)$ of zero-cycles (the kernel of the intersection form) is invariant, and
  $\pi_1(\C \setminus  \mathcal A ,*)$ acts on it trivially.
Therefore the reduced  representation of the fundamental group on $V=H_1(f^{-1}(t), \Z)/V_0$ is well defined too, and $V$ carries an invariant \emph{ non-degenerate} intersection form. 
The reduced monodromy group is thus a subgroup of  $ Sp(2p,\C)$ and denote by  $\mathbf G$ its Zarisky closure. Here $2p=\dim V$ and $p$ is the genus of the Riemann surface of $f^{-1}(t)$, $t\not\in \mathcal A$.

 It is well known that for  generic $f$ (e.g. Morse plus polynomials) we  have $\mathbf G = Sp(2p,\C)$. According to Theorem \ref{case2} this holds true also in the special Lotka-Volterra case $f=x^py^p(1-x-y)$, $p=2$.
\emph{ We conjecture that $\mathbf G = Sp(2p,\C)$ for every integer $p\geq 1$.}

On the other hand, if $f$ is a composite polynomial,  $f=g\circ h$, where $h : \C^2 \to \C^2$ is a polynomial mapping,
we can not expect that $\mathbf G = Sp(2p,\C)$. A simple example is $f= y^2+ P(x^2)$ for $P$ a polynomial of degree at least three. The  natural involution $x\to -x$ induces a decomposition $H_1(f^{-1}(t), \Z)/V_0 = V_+\oplus V_-$ where $V_{\pm}$ are invariant under $\mathbf G$ so  $\mathbf G \neq Sp(2p,\C)$. We note that  by the Ritt theorem \cite{ritt22} a univariate polynomial $f\in \C[x]$ is composite if and only if its monodromy group is imprimitive. Are there examples of non-composite  bivariate polynomials $f=f(x,y)$, such that $\mathbf G \neq Sp(2p,\C)$ ? 

\def\cprime{$'$} \def\cprime{$'$} \def\cprime{$'$} \def\cprime{$'$}
  \def\cprime{$'$} \def\cprime{$'$} \def\cprime{$'$}

\end{document}

%% file: dessin2rev.pdf_tex
\begingroup%
  \makeatletter%
  \providecommand\color[2][]{%
    \errmessage{(Inkscape) Color is used for the text in Inkscape, but the package 'color.sty' is not loaded}%
    \renewcommand\color[2][]{}%
  }%
  \providecommand\transparent[1]{%
    \errmessage{(Inkscape) Transparency is used (non-zero) for the text in Inkscape, but the package 'transparent.sty' is not loaded}%
    \renewcommand\transparent[1]{}%
  }%
  \providecommand\rotatebox[2]{#2}%
  \ifx\svgwidth\undefined%
    \setlength{\unitlength}{374.21057129bp}%
    \ifx\svgscale\undefined%
      \relax%
    \else%
      \setlength{\unitlength}{\unitlength * \real{\svgscale}}%
    \fi%
  \else%
    \setlength{\unitlength}{\svgwidth}%
  \fi%
  \global\let\svgwidth\undefined%
  \global\let\svgscale\undefined%
  \makeatother%
  \begin{picture}(1,1.03663735)%
    \put(0,0){\includegraphics[width=\unitlength,page=1]{dessin2rev.pdf}}%
    \put(0.45949204,0.82255195){\color[rgb]{0,0,0}\makebox(0,0)[lb]{\smash{$\gamma_0$}}}%
    \put(0.64634085,0.64468481){\color[rgb]{0,0,0}\makebox(0,0)[lb]{\smash{$\gamma_1$}}}%
    \put(0.50340882,0.34510842){\color[rgb]{0,0,0}\makebox(0,0)[lb]{\smash{$\gamma_2$}}}%
    \put(0,0){\includegraphics[width=\unitlength,page=2]{dessin2rev.pdf}}%
    \put(0.16944125,0.47756234){\color[rgb]{0,0,0}\makebox(0,0)[lb]{\smash{$\gamma_3$}}}%
    \put(0,0){\includegraphics[width=\unitlength,page=3]{dessin2rev.pdf}}%
  \end{picture}%
\endgroup%

%% file: dessin4.pdf_tex
\begingroup%
  \makeatletter%
  \providecommand\color[2][]{%
    \errmessage{(Inkscape) Color is used for the text in Inkscape, but the package 'color.sty' is not loaded}%
    \renewcommand\color[2][]{}%
  }%
  \providecommand\transparent[1]{%
    \errmessage{(Inkscape) Transparency is used (non-zero) for the text in Inkscape, but the package 'transparent.sty' is not loaded}%
    \renewcommand\transparent[1]{}%
  }%
  \providecommand\rotatebox[2]{#2}%
  \ifx\svgwidth\undefined%
    \setlength{\unitlength}{401.65224609bp}%
    \ifx\svgscale\undefined%
      \relax%
    \else%
      \setlength{\unitlength}{\unitlength * \real{\svgscale}}%
    \fi%
  \else%
    \setlength{\unitlength}{\svgwidth}%
  \fi%
  \global\let\svgwidth\undefined%
  \global\let\svgscale\undefined%
  \makeatother%
  \begin{picture}(1,0.98168971)%
    \put(0,0){\includegraphics[width=\unitlength,page=1]{dessin4.pdf}}%
    \put(0.95741421,0.15617817){\color[rgb]{0,0,0}\makebox(0,0)[lb]{\smash{$x$}}}%
    \put(0.12269421,0.96353033){\color[rgb]{0,0,0}\makebox(0,0)[lb]{\smash{$y$}}}%
  \end{picture}%
\endgroup%

%% file: dessin3.pdf_tex
\begingroup%
  \makeatletter%
  \providecommand\color[2][]{%
    \errmessage{(Inkscape) Color is used for the text in Inkscape, but the package 'color.sty' is not loaded}%
    \renewcommand\color[2][]{}%
  }%
  \providecommand\transparent[1]{%
    \errmessage{(Inkscape) Transparency is used (non-zero) for the text in Inkscape, but the package 'transparent.sty' is not loaded}%
    \renewcommand\transparent[1]{}%
  }%
  \providecommand\rotatebox[2]{#2}%
  \ifx\svgwidth\undefined%
    \setlength{\unitlength}{435.64169095bp}%
    \ifx\svgscale\undefined%
      \relax%
    \else%
      \setlength{\unitlength}{\unitlength * \real{\svgscale}}%
    \fi%
  \else%
    \setlength{\unitlength}{\svgwidth}%
  \fi%
  \global\let\svgwidth\undefined%
  \global\let\svgscale\undefined%
  \makeatother%
  \begin{picture}(1,0.34412812)%
    \put(0.05842162,0.41294245){\color[rgb]{0,0,0}\makebox(0,0)[lb]{\smash{}}}%
    \put(0.1889997,-0.07053638){\color[rgb]{0,0,0}\makebox(0,0)[lb]{\smash{}}}%
    \put(0.55627396,-0.07053638){\color[rgb]{0,0,0}\makebox(0,0)[lb]{\smash{}}}%
    \put(0,0){\includegraphics[width=\unitlength,page=1]{dessin3.pdf}}%
  \end{picture}%
\endgroup%

%% file: dessin5.pdf_tex
\begingroup%
  \makeatletter%
  \providecommand\color[2][]{%
    \errmessage{(Inkscape) Color is used for the text in Inkscape, but the package 'color.sty' is not loaded}%
    \renewcommand\color[2][]{}%
  }%
  \providecommand\transparent[1]{%
    \errmessage{(Inkscape) Transparency is used (non-zero) for the text in Inkscape, but the package 'transparent.sty' is not loaded}%
    \renewcommand\transparent[1]{}%
  }%
  \providecommand\rotatebox[2]{#2}%
  \ifx\svgwidth\undefined%
    \setlength{\unitlength}{670.91450195bp}%
    \ifx\svgscale\undefined%
      \relax%
    \else%
      \setlength{\unitlength}{\unitlength * \real{\svgscale}}%
    \fi%
  \else%
    \setlength{\unitlength}{\svgwidth}%
  \fi%
  \global\let\svgwidth\undefined%
  \global\let\svgscale\undefined%
  \makeatother%
  \begin{picture}(1,0.36665356)%
    \put(0,0){\includegraphics[width=\unitlength,page=1]{dessin5.pdf}}%
    \put(0.80011575,0.29225685){\color[rgb]{0,0,0}\makebox(0,0)[lb]{\smash{$\delta_1$}}}%
    \put(0.6816509,0.17767612){\color[rgb]{0,0,0}\makebox(0,0)[lb]{\smash{$\delta_0$}}}%
    \put(0,0){\includegraphics[width=\unitlength,page=2]{dessin5.pdf}}%
  \end{picture}%
\endgroup%

%% file: dessin1.pdf_tex
\begingroup%
  \makeatletter%
  \providecommand\color[2][]{%
    \errmessage{(Inkscape) Color is used for the text in Inkscape, but the package 'color.sty' is not loaded}%
    \renewcommand\color[2][]{}%
  }%
  \providecommand\transparent[1]{%
    \errmessage{(Inkscape) Transparency is used (non-zero) for the text in Inkscape, but the package 'transparent.sty' is not loaded}%
    \renewcommand\transparent[1]{}%
  }%
  \providecommand\rotatebox[2]{#2}%
  \ifx\svgwidth\undefined%
    \setlength{\unitlength}{382.96237735bp}%
    \ifx\svgscale\undefined%
      \relax%
    \else%
      \setlength{\unitlength}{\unitlength * \real{\svgscale}}%
    \fi%
  \else%
    \setlength{\unitlength}{\svgwidth}%
  \fi%
  \global\let\svgwidth\undefined%
  \global\let\svgscale\undefined%
  \makeatother%
  \begin{picture}(1,0.58698402)%
    \put(0,0){\includegraphics[width=\unitlength,page=1]{dessin1.pdf}}%
    \put(-0.05617276,0.5565028){\color[rgb]{0,0,0}\makebox(0,0)[lb]{\smash{}}}%
    \put(0.09236732,0.00651784){\color[rgb]{0,0,0}\makebox(0,0)[lb]{\smash{}}}%
    \put(0.09236732,0.00651784){\color[rgb]{0,0,0}\makebox(0,0)[lb]{\smash{$p=1$}}}%
    \put(0.44749359,0.00651784){\color[rgb]{0,0,0}\makebox(0,0)[lb]{\smash{$p=2$}}}%
    \put(0.78173008,0.00651784){\color[rgb]{0,0,0}\makebox(0,0)[lb]{\smash{$p=3$}}}%
    \put(0.51016293,0.00651784){\color[rgb]{0,0,0}\makebox(0,0)[lb]{\smash{}}}%
    \put(0.44749359,0.13185653){\color[rgb]{0,0,0}\makebox(0,0)[lb]{\smash{$\delta_0$}}}%
    \put(0.46838337,0.42431348){\color[rgb]{0,0,0}\makebox(0,0)[lb]{\smash{$\delta_1$}}}%
    \put(0.48927315,0.36164414){\color[rgb]{0,0,0}\makebox(0,0)[lb]{\smash{$\delta_{12}^0$}}}%
    \put(0.48927315,0.50787261){\color[rgb]{0,0,0}\makebox(0,0)[lb]{\smash{}}}%
    \put(0.48927315,0.52876239){\color[rgb]{0,0,0}\makebox(0,0)[lb]{\smash{$\delta_{12}^1$}}}%
    \put(0.36393447,0.09007697){\color[rgb]{0,0,0}\makebox(0,0)[lb]{\smash{$\delta_{13}$}}}%
    \put(0.59372205,0.09007697){\color[rgb]{0,0,0}\makebox(0,0)[lb]{\smash{$\delta_{23}$}}}%
  \end{picture}%
\endgroup%